\documentclass[a4paper,11pt]{article}
\usepackage{pdfsync}
\usepackage{amsmath}
\usepackage{amsthm}
\usepackage{amsfonts}
\usepackage{amssymb}
\usepackage[all,2cell]{xy} \UseAllTwocells
\usepackage{enumerate}
\usepackage{setspace}

\newtheorem{thm}{Theorem}[section]

\newtheorem{lem}[thm]{Lemma}
\newtheorem{rem}[thm]{Remark}
\newtheorem{cor}[thm]{Corollary}
\newtheorem{con}[thm]{Condition}
\newtheorem{prop}[thm]{Proposition}

\newcommand{\PtV}{\mathbf{Pt}(\V)}
\newcommand{\Pt}[1]{\mathbf{Pt}(\mathbb{#1})}
\newcommand{\Pts}[1]{\mathbf{Pt}(#1)}
\newcommand{\BB}[1]{\mathbb{#1}}
\newcommand{\BF}[1]{\mathbf{#1}}
\newcommand{\SE}[1]{\mathbf{SplExt}(#1)}
\newcommand{\V}{\cal{V}}
\newcommand{\Set}{\mathbf{Set}}
\newcommand{\pb}[2]{\substack{ \hole  \\ \times\\\langle #1,#2\rangle}}
\newcommand{\ti}[1]{\tilde{#1}}

\title{ON ALGEBRAIC AND MORE GENERAL CATEGORIES WHOSE SPLIT EPIMORPHISMS HAVE UNDERLYING PRODUCT PROJECTIONS}
\author{J. R. A. GRAY AND N. MARTINS-FERREIRA}
\begin{document}
\maketitle
\begin{abstract}
We characterize those varieties of universal algebras where every split epimorphism considered as a map of sets is a product projection.  In addition we obtain new characterizations of protomodular, unital and subtractive varieties as well as varieties of right $\Omega$-loops and biternary systems.
\end{abstract}
\section*{Introduction}
It is well known that in the category of groups if
\[\xymatrix{
0\ar[r] & K \ar[r] ^{\kappa} & A\ar[r]^{\alpha} & B\ar[r] & 0
}
\]
is a short exact sequence, then $A$ and $K\times B$ are bijective as sets, moreover when $\alpha$ is split, i.e. for each split extension 
\[\xymatrix{ K \ar[r]^{\kappa} & A \ar@<0.5ex>[r]^{\alpha} & B\ar@<0.5ex>[l]^{\beta}},\ \alpha\beta=1_B,\ \kappa = \ker{(\alpha)},\]
this bijection becomes a natural bijection $K\times B\to A$ such that the diagram
\[
\xymatrix{
K\ar@{=}[d] \ar[r]^{\langle 1,0\rangle} & K\times B\ar[d]^{\varphi} \ar@<0.5ex>[r]^{\pi_2} & B\ar@{=}[d]\ar@<0.5ex>[l]^{\langle 0,1\rangle}\\ 
 K \ar[r]^{\kappa} & A \ar@<0.5ex>[r]^{\alpha} & B\ar@<0.5ex>[l]^{\beta}}
\]
is a morphism of split extensions in the category $\BF{Set}$, of sets, that is, $\alpha\varphi=\pi_2$, $\varphi\langle 0,1\rangle=\beta$, and $\varphi \langle 1,0\rangle = \kappa$.  As shown by E. B. Inyangala, these bijections exists in a more general setting of a variety of right $\Omega$-loops (see \cite{Inyangala 2010,Inyangala 2011}), that is, a pointed variety of universal algebras $\V$ with constant $0$ and binary terms $x+y$ and $x-y$ satisfying the identities:
\begin{eqnarray}
x + 0 = x\\
x-x = 0\\
(x+y)-y=x\\
(x-y)+y=x
\end{eqnarray}
Moreover, he showed that if a pointed variety $\V$ with constant $0$ has binary terms $x+y$ and $x-y$ and there exist bijections (as above) constructed (in the same way as for groups) using those terms, i.e. $\varphi(k,b)=\kappa(k)+\beta(b)$ and $\varphi^{-1}(a)=(\lambda(a),\alpha(a))$, where $\lambda$ is the unique map such that $\kappa\lambda(a)=a-\beta\alpha(a)$ , then $\V$ is a variety of right $\Omega$-loops and in particular the identities (1) - (4) hold for $x+y$ and $x-y$.  In this paper we prove that if for a pointed variety $\V$ there exist natural bijections as above, then $\V$ is a variety of right $\Omega$-loops (see Theorem \ref{pointed variety classification}).\\[10pt]
For any category $\BB{C}$ let $\Pt{C}$ to be the category of split epimorphisms in $\mathbb{C}$: an object is a quadruple $(A,B,\alpha,\beta)$ where $A$ and $B$ are objects in $\BB{C}$ and $\alpha :A\to B$ and $\beta : B\to A$ are morphisms in $\BB{C}$ with $\alpha\beta =1_B$; a morphism $(A,B,\alpha,\beta) \to (A',B',\alpha',\beta')$ is a pair of morphisms $(f:A\to A',g:B\to B')$ such that in the diagram
\[
\xymatrix{
A\ar@<0.5ex>[r]^{\alpha}\ar[d]_{f} & B\ar@<0.5ex>[l]^{\beta} \ar[d]^{g}\\
A'\ar@<0.5ex>[r]^{\alpha'} & B'\ar@<0.5ex>[l]^{\beta'} 
}
\]
$\alpha'f=g\alpha$ and $f\beta=\beta'g$.
Throughout this paper for any objects $A$ and $B$ we will denote by $\pi_1$ and $\pi_2$ the first and second product projections respectively.  We will use the same notation for the first and second pullback projections and will write \[(A\pb{f}{g} B,\pi_1,\pi_2)\] for the pullback of $f:A\to C$ and $g:B\to C$ as in the diagram
\[
\xymatrix{
A \pb{f}{g} B \ar[r]^{\pi_2}\ar[d]_{\pi_1} & B\ar[d]^{g}\\
A \ar[r]_{f} & C.
}
\]
For any morphisms $u:W\to A$ and $v: W\to B$ with $fu=gv$ we will write \[\langle u,v\rangle : W\to A\pb{f}{g} B\] for the unique morphism with $\pi_1\langle u,v\rangle=u$ and $\pi_2\langle u,v\rangle=v$.\\[10pt]
We prove that for a pointed variety $\V$, if for each $(A,B,\alpha,\beta)$ in $\PtV$ there exists a natural bijection $\varphi: K\times B \to A$, where $\kappa:K\to A$ is the kernel of $\alpha$,  such that the diagram
\[
\xymatrix{
K\times B \ar@<0.5ex>[r]^{\pi_2}\ar[d]_{\varphi} & B\ar@{=}[d]\ar@<0.5ex>[l]^{\langle 0,1\rangle}\\
A\ar@<0.5ex>[r]^{\alpha}& B\ar@<0.5ex>[l]^{\beta}
}
\]
is a morphism in $\Pts{\BF{Set}}$, then $\V$ is a variety of right $\Omega$-loops (see Corollary \ref{main classification of right omega loops} ).
There is a \emph{natural} generalization of this condition for any variety $\V$, namely asking for each $(A,B,\alpha,\beta)$ in $\PtV$ and for each morphism $f:E\to B$ that there exists a bijection \[\varphi: (A\pb{\alpha}{f}E)\times B \to E \times A \] natural in both $(A,B,\alpha,\beta)$ and $f:E\to B$,
such that the diagram
\[
\xymatrix{
(A\pb{\alpha}{f}E)\times B \ar@<0.5ex>[r]^{\pi_2\times 1}\ar[d]_{\varphi} & E\times B \ar@<0.5ex>[l]^{\langle \beta f,1\rangle \times 1}\ar@{=}[d]\\
E\times A \ar@<0.5ex>[r]^{1\times \alpha} & E\times B \ar@<0.5ex>[l]^{1\times \beta}
}
\]
is a morphism in $\Pts{\BF{Set}}$. It is clear that for a pointed variety this condition implies the previous condition, since taking $E$ to be the zero object and $f$ to be the unique morphism from $E$ to $B$ makes
\[\pi_1:A\pb{\alpha}{f}E \to A \]
the kernel of $\alpha$. In Section \ref{general varieties} we prove that this condition is equivalent to the same condition under the restriction that each $f$ as above is an identity morphism (see Theorem \ref{generalization of introduction equivalent formulations}).  We also prove that a variety satisfies this condition if and only if it is a \emph{biternary system} \cite{Mal'tsev 1954} that is there exist ternary terms $p(x,y,z)$ and $q(x,y,z)$ satisfying the identities
\begin{eqnarray}
p(x,x,y)=y\\
p(q(x,y,z),z,y)=x=q(p(x,y,z),z,y).
\end{eqnarray}
However, there are other generalizations that may be considered.  In a variety $\V$ with constants, for each $X$, let $\theta_X: 1\to X^n$ be a map (natural in $X$) such that the composite with each product projection $\pi_i:X^n\to X$ gives a constant. We could then consider the following condition: for each $(A,B,\alpha,\beta)$ in $\PtV$ there exists a natural split epimorphism (in the category of sets) \[\varphi : (A^n\pb{\alpha^n}{\theta_B} 1) \times B \to A\] with splitting \[
\psi : A \to (A^n\pb{\alpha^n}{\theta_B} 1) \times B\]
such that in the diagram
\[
\xymatrix@C=12ex{
(A^n\pb{\alpha^n}{\theta_B} 1)\times B \ar@<0.5ex>[d]^{\varphi} \ar@<0.5ex>[r]^{\pi_2} & B \ar@<0.5ex>[l]^{(\langle \theta_A,1\rangle !_B) \times 1}\ar@{=}[d]\\
A \ar@<0.5ex>[r]^{\alpha}\ar@<0.5ex>[u]^{\psi} & B\ar@<0.5ex>[l]^{\beta}
}
\]
 the upward and downward directed sub-diagrams are morphisms in $\Pts{\BF{Set}}$.  We prove in Section \ref{varieties with constants} that this condition is equivalent to $\V$ being a protomodular variety \cite{Bourn Janelidze 2003} of \emph{type} $n$, that is, a variety $\V$ with constants $e_1, \dots, e_n$, binary terms $s_1(x,y),\dots, s_n(x,y)$ and an $n+1$-ary term $p(x_1,\dots,x_n,z)$ satisfying the identities:
\begin{eqnarray}
s_i(x,x)=e_i \ i\in \{1,\dots,n\} \\
p(s_1(x,z),\dots,s_n(x,z),z)=x.
\end{eqnarray}
Note that requiring $\varphi$ to be a bijection gives the addition conditions
\begin{eqnarray}
s_i(p(x_1,...,x_n,y),y))=x_i\textnormal{ for all }i \in {1,...,n}.
\end{eqnarray}
In order to study these conditions simultaneously we make a further generalization described in Section 1.
\section{The general setting}
\label{general theory}
In this section we replace a forgetful functor from a variety into the category of sets (or pointed sets) with an abstract functor (satisfying certain conditions) and consider a generalization allowing us to study simultaneously both generalizations discussed in the introduction.

For a set $\BF{n}$, a category $\BB{D}$ with finite products and products indexed over $\BF{n}$, and for functors $F,G,H:\BB{C}\to \BB{D}$ we denote by $F^\BF{n}$ the $\BF{n}$ indexed product of $F$ with itself and by $G\times H$  the product of $G$ and $H$ in the functor category $\BB{D} ^\BB{C}$.\\[10pt]
Throughout this section we will assume that:
\begin{enumerate}
\item $\mathbb{A}$ is a category with finite products;
\item $\BF{m}$ and $\BF{n}$ are sets;
\item $\mathbb{X}$ is a category with finite limits and products indexed by the sets $\BF{m}$ and $\BF{n}$;
\item $U: \mathbb{A}\to \mathbb{X}$ is a functor preserving finite products;
\item $\theta:U^\BF{m}\to U^\BF{n}$ is a natural transformation.
\end{enumerate}
Let $\Delta : \BB{A} \to \Pt{A}$ be the functor sending $X$ in $\BB{A}$ to $(X\times X,X,\pi_2,\langle 1,1\rangle)$ and let $D_\BB{A}$ be the functor $\Pt{A} \to \BB{A}$ taking $(A,B,\alpha,\beta)$ to $B$.
Let $V: \mathbf{Pt}(\mathbb{A}) \to \Pt{X}$ and $W:\Pt{A} \to \Pt{X}$ be the functors sending $(A,B,\alpha,\beta)$ in $\mathbf{Pt}(\mathbb{A})$ to 
\[((U(A)^\BF{n}\pb{U(\alpha)^\BF{n}}{\theta_B} U(B)^\BF{m})\times U(B),U(B)^\BF{m}\times U(B),\pi_2\times 1,\langle U(\beta)^\BF{n}\theta_B, 1\rangle\times 1)\]
and
\[( U(B)^\BF{m} \times U(A),U(B)^\BF{m} \times U(B),1\times U(\alpha),1\times U(\beta) )\]
respectively.\\
From the beginning of the next section we will consider the case where $\mathbb{A}$ is a variety, $\mathbb{X}$ is the category of sets, $U$ is the usual forgetful functor from the variety to the category of sets, $\BF{m}=\{1,\dots,m\}$, $\BF{n}=\{1,\dots,n\}$, and $\theta$ is constructed from $n$ $m$-ary terms of $\mathbb{A}$.  In particular when $\mathbb{A}$ is pointed with constant $0$, $\BF{n}=\{1\}$, $\BF{m}=\emptyset$, and $\theta : U^{\BF{m}}\to U^{\BF{n}}$ is the natural transformation with component at $X$ $\theta_X(1)=0$ (where $1$ is the unique element in $U^{\BF{m}}(X)$), it can be seen that \[\pi_1:U(A)^\BF{n}\pb{U(\alpha)^\BF{n}}{\theta_B} U(B)^\BF{m}\to U(A)\] is up to isomorphism the image under $U$ of the kernel of $\alpha$ and the bijections mentioned at the start of the introduction become components of a natural transformation $V\to W$.
\begin{lem}
Each of the following types of data uniquely determine each other:
\begin{enumerate}[(a)]
\item a natural transformation $\tau:V \to W$;
\item a natural transformation $\overline{\tau}:V\Delta \to W\Delta$;
\item natural transformations $\rho:(U^\BF{n}\times U^\BF{m})\times U \to U$ and $\zeta : U^\BF{m}\times U \to U^\BF{m}$;
\end{enumerate}
\end{lem}
\begin{proof}
For each $(A,B,\alpha,\beta)$ in $\Pt{A}$ and $X$ in $\BB{A}$, let $(\varphi_{1_{(A,B,\alpha,\beta)}},\varphi_{0_{(A,B,\alpha,\beta)}})=\tau_{(A,B,\alpha,\beta)}$ and $(\overline{\varphi}_{1_{X}},\overline{\varphi}_{0_{X}})=\overline{\tau}_X$.
The diagram 
\[
\xymatrix@C=-4 ex{
& P_X\times U(X) \ar[dl]_{\langle U(\pi_1)^\BF{n}\pi_1,\pi_2\rangle \times 1\ \ }\ar@<0.5ex>[ddd]|(0.35){\hole}|(0.4){\hole}|(0.42){\hole}^{\pi_2 \times 1}\ar[rrrrr]^{\overline{\varphi}_{1_{X}}} & & & & & U(X)^\BF{m} \times U(X\times X)\ar[dllll]_{1\times \langle U(\pi_1),U(\pi_2)\rangle\ }\ar@<0.5ex>[ddd]|(0.32){\hole}|(0.37){\hole}|(0.41){\hole}^(0.5){1\times U(\pi_2)}\\
(U(X)^\BF{n} \times U(X)^\BF{m}) \times U(X) \ar@/_2ex/@<0.5ex>[rdd]^{\pi_2\times 1} \ar[rr]_{p_X} & & U(X)^\BF{m}\times (U(X)\times U(X))\ar@/_2ex/@<0.5ex>[rrrrdd]^{1\times \pi_2} & & & & & \\
\\
& U(X)^\BF{m}\times U(X) \ar@<0.5ex>[uuu]^{\langle U\langle 1,1\rangle^\BF{n}\theta_X,1\rangle \times 1}|(0.58){\hole}|(0.6){\hole}|(0.65){\hole}\ar[rrrrr]_{\overline{\varphi}_{0_{X}}} \ar@/^2ex/@<0.5ex>[luu]^{\langle \theta_X,1\rangle \times 1}& & & & & U(X)^\BF{m}\times U(X),\ar@<0.5ex>[uuu]^(0.5){1\times U\langle 1,1\rangle}|(0.59){\hole}|(0.63){\hole}|(0.68){\hole}\ar@/^2ex/@<0.5ex>[lllluu]^{1\times \langle 1,1\rangle}
}
\]
in which \[P_X=U(X\times X)^\BF{n}\pb{U(\pi_2)^n}{\theta_X} U(X)^\BF{m}\] and  \[p_X =\langle \zeta_X(\pi_2\times 1),\langle \rho_X,\rho_X (\langle \theta_X\pi_2,\pi_2\rangle\times 1)\rangle \rangle,\]
is a commutative diagram of morphisms in $\Pt{X}$, and shows the relationship between $\overline{\tau}$ and $\rho$ and $\zeta$. The commutative diagrams \[
\xymatrix@C=15ex{
(U(A)^\BF{n}\pb{U(\alpha)^\BF{n}}{\theta_B} U(B)^\BF{m})\times U(B)\ar`l[dd] `[dd]^(0.25){\langle \pi_1,U(\beta)^\BF{m}\pi_2\rangle \times U(\beta)} [dd] \ar[r]^(0.6){\varphi_{1_{(A,B,\alpha,\beta)}}}\ar[d]^{(U(\langle 1,\beta\alpha\rangle)^\BF{n}\times U(\beta)^\BF{m})\times U(\beta)} & U(B)^\BF{m}\times U(A)\ar[d]_{U(\beta)^\BF{m}\times U(\langle 1,\beta\alpha\rangle)}\ar@<0.2ex>@{-}`r[dd] `[dd] [dd]\ar@<-0.2ex>@{-}`r[dd] `[dd] [dd]\\
(U(A\times A)^\BF{n}\pb{U(\pi_2)^\BF{n}}{\theta_A} U(A)^\BF{m})\ar[r]_(0.6){\varphi_{1_{\Delta(A)}}=\overline{\varphi}_{1_A}} \times U(A)\ar[d]^{\langle U(\pi_1)^\BF{n}\pi_1,\pi_2\rangle\times 1}& U(A)^\BF{m}\times U(A\times A)\ar[d]_{U(\alpha)^\BF{m}\times U(\pi_1)}\\
(U(A)^\BF{n}\times U(A)^\BF{n})\times U(A)\ar[r]_{\langle U(\alpha)^\BF{m}\zeta_A(\pi_2\times 1),\rho_A\rangle}& U(B)^\BF{m}\times U(A)
}
\]
\[
\xymatrix@C=15ex
{
U(B)^\BF{m}\times U(B) \ar[r]^{\varphi_{0_{(A,B,\alpha,\beta)}}} \ar[d]^{U(\beta)^\BF{m}\times U(\beta)} \ar@<0.2ex>@{-}`l[dd] `[dd] [dd]\ar@<-0.2ex>@{-}`l[dd] `[dd] [dd] & U(B^m\times U(B))\ar[d]_{U(\beta)^\BF{m}\times U(\beta)} \ar@<0.2ex>@{-}`r[dd] `[dd] [dd]\ar@<-0.2ex>@{-}`r[dd] `[dd] [dd]\\
U(A)^m\times U(A)\ar[r]^{\varphi_{0_{\Delta(A)}}} \ar[d]^{U(\alpha)^\BF{m}\times U(\alpha)}& U(A)^\BF{m}\times U(A)\ar[d]_{U(\alpha)^\BF{m}\times U(\alpha)} \\
U(B)^\BF{m}\times U(B)\ar[r]^{\varphi_{0_{\Delta(B)}}} \ar[dr]_{\langle \zeta_B,\rho_B(\langle \theta_B,1\rangle \times 1) \rangle}& U(B)^\BF{m}\times U(B) \ar@{=}[d] \\
& U(B)^\BF{m}\times U(B)
}
\]
show the relationships between $\tau$ and $\overline{\tau}$, and $\tau$ and $\rho$ and $\zeta$.
\end{proof}
\begin{lem}
Each of the following types of data uniquely determine each other:
\begin{enumerate}[(a)]
\item a natural transformation $\gamma: W\to V$;
\item a natural transformation $\overline{\gamma} : W\Delta \to V\Delta$;
\item natural transformations $\sigma : U^\BF{m}\times (U\times U) \to U^\BF{n}$, $\eta : U^\BF{m}\times U \to U^\BF{m}$ and $\epsilon : U^\BF{m}\times U \to U$ with components at each $X$ in $\BB{A}$ making the diagram
\begin{equation}
\label{axiom1}
\vcenter{
\xymatrix{
U(X)^\BF{m}\times U(X) \ar[r]^{1\times\langle 1,1\rangle} \ar[d]_{\eta_X} & U(X)^\BF{m}\times (U(X)\times U(X)) \ar[d]^{\sigma_X} \\
U(X)^\BF{m} \ar[r]_{\theta_X} & U(X)^\BF{n}
}
}
\end{equation}
commute.
\end{enumerate}
\end{lem}
\begin{proof}
For each $(A,B,\alpha,\beta)$ in $\Pt{A}$ and $X$ in $\BB{A}$, let $(\psi_{1_{(A,B,\alpha,\beta)}},\psi_{0_{(A,B,\alpha,\beta)}})=\gamma_{(A,B,\alpha,\beta)}$ and $(\overline{\psi}_{1_{X}},\overline{\psi}_{0_{X}})=\overline{\gamma}_X$.  The diagram
\[
\xymatrix@C=-4 ex{
& P_X\times U(X) \ar[dl]_{\langle U(\pi_1)^\BF{n}\pi_1,\pi_2\rangle \times 1\ \ }\ar@<0.5ex>[ddd]|(0.35){\hole}|(0.4){\hole}|(0.42){\hole}^{\pi_2 \times 1} & & & & & U(X)^\BF{m} \times U(X\times X)\ar[lllll]_{\overline{\psi}_{1_{X}}}\ar[dllll]_{1\times \langle U(\pi_1),U(\pi_2)\rangle\ }\ar@<0.5ex>[ddd]|(0.32){\hole}|(0.37){\hole}|(0.41){\hole}^(0.5){1\times U(\pi_2)}\\
(U(X)^\BF{n} \times U(X)^\BF{m}) \times U(X) \ar@/_2ex/@<0.5ex>[rdd]^{\pi_2\times 1}  & & U(X)^\BF{m}\times (U(X)\times U(X))\ar[ll]^{q_X}\ar@/_2ex/@<0.5ex>[rrrrdd]^{1\times \pi_2} & & & & & \\
\\
& U(X)^\BF{m}\times U(X) \ar@<0.5ex>[uuu]^{\langle U\langle 1,1\rangle^\BF{n}\theta_X,1\rangle \times 1}|(0.58){\hole}|(0.6){\hole}|(0.65){\hole} \ar@/^2ex/@<0.5ex>[luu]^{\langle \theta_X,1\rangle \times 1}& & & & & U(X)^\BF{m}\times U(X),\ar[lllll]^{\overline{\psi}_{0_{X}}}\ar@<0.5ex>[uuu]^(0.5){1\times U\langle 1,1\rangle}|(0.59){\hole}|(0.63){\hole}|(0.68){\hole}\ar@/^2ex/@<0.5ex>[lllluu]^{1\times \langle 1,1\rangle}
}
\]
in which 
\[P_X=U(X\times X)^\BF{n}\pb{U(\pi_2)^n}{\theta_X} U(X)^\BF{m}\] and \[q_X=\langle \langle \sigma_X,\eta_X(1\times \pi_2)\rangle,\epsilon_X(1\times \pi_2)\rangle,\] is a commutative diagram of morphisms in $\Pt{X}$, and shows the relationship between $\overline{\gamma}$ and $\sigma$, $\eta$ and $\epsilon$.  The equations \[\gamma_{\Delta(X)}=\overline{\gamma}_X\]
and
\[\psi_{1_{(A,B,\alpha,\beta)}}=\langle \langle  \sigma_A (U(\beta)^\BF{m}\times U(\langle 1,\beta\alpha\rangle)),\eta_B(1\times U(\alpha))\rangle \epsilon_B (1\times U(\alpha))\rangle, \]
and the commutative diagram
\[
\xymatrix@C=15ex
{
U(B)^\BF{m}\times U(B)  \ar[d]^{U(\beta)^\BF{m}\times U(\beta)} \ar@<0.2ex>@{-}`l[dd] `[dd] [dd]\ar@<-0.2ex>@{-}`l[dd] `[dd] [dd] & U(B)^\BF{m}\times U(B)\ar[l]_{\psi_{0_{(A,B,\alpha)}}}\ar[d]_{U(\beta)^\BF{m}\times U(\beta)} \ar@<0.2ex>@{-}`r[dd] `[dd] [dd]\ar@<-0.2ex>@{-}`r[dd] `[dd] [dd]\\
U(A)^\BF{m}\times U(A) \ar[d]^{U(\alpha)^\BF{m}\times U(\alpha)}& U(A)^\BF{m}\times U(A)\ar[l]_{\psi_{0_{\Delta(A)}}} \ar[d]_{U(\alpha)^\BF{m}\times U(\alpha)} \\
U(B)^\BF{m}\times U(B) \ar@{=}[d]&  U(B)^m\times U(B) \ar[l]_{\psi_{0_{\Delta(B)}}} \ar[dl]^{\langle \eta_B,\epsilon_B \rangle} \\
U(B)^\BF{m}\times U(B) & 
}
\]
show the relationships between $\gamma$ and $\overline{\gamma}$, and $\gamma$ and $\sigma$, $\eta$ and $\epsilon$. 
\end{proof}

From the two lemmas above we easily prove the following corollaries.
\begin{cor}
\label{thm varphi and rho}
Each of the following types of data uniquely determine each other:
\begin{enumerate}[(a)]
\item \label{V to W}
a natural transformation $\tau:V\to W$ with $1_{D_{\BB{X}}}\circ \tau=1_{D_\BB{A}^\BF{m}\times D_\BB{A}}$;
\item \label{generalized addition}
a natural transformation $\rho:(U^\BF{n}\times U^\BF{m}) \times U \to U$ with component at each $X$ in $\BB{C}$ making the diagram
\begin{equation}
\label{generalized addition axiom}
\vcenter{
\xymatrix{
(U(X)^\BF{n}\times U(X)^\BF{m})\times U(X)\ar[r]^(0.7){\rho_X}&U(X)\\
U(X)^\BF{m}\times U(X)\ar[ru]_{\pi_2}\ar[u]^{\langle \theta_{X} ,1 \rangle \times 1}&
}
}
\end{equation}
 commute.
\end{enumerate}
\end{cor}
\begin{cor}
\label{thm psi and sigma}
Each of the following types of data uniquely determine each other:
\begin{enumerate}[(a)]
\item \label{W to V}
a natural transformation $\gamma:W\to V$ with $1_{D_{\BB{X}}}\circ \gamma=1_{D_\BB{A}^\BF{m}\times D_\BB{A}}$;
\item \label{generalized subtractions}
a natural transformation $\sigma: U^\BF{m}\times (U\times U)  \to U^\BF{n}$ with component at each $X$ in $\BB{C}$ making the diagram
\begin{equation}
\label{generalized subtractions axiom}
\vcenter{
\xymatrix{
U(X)^\BF{m} \times (U(X)\times U(X)) \ar[r]^(0.7){\sigma_X} &U(X)^\BF{n}\\
U(X)^\BF{m}\times U(X) \ar[r]_{\pi_1} \ar[u]^{1\times \langle 1,1\rangle}& U(X)^\BF{m} \ar[u]_{\theta_X} 
}
}
\end{equation}
 commute.
\end{enumerate}
\end{cor}
\begin{cor}
\label{protomodular cor}
Each of the following types of data uniquely determine each other:
\begin{enumerate}[(a)]
\item natural transformations $\tau:V\to W$ and $\gamma:W\to V$ with $1_{D_{\BB{X}}}\circ \tau=1_{D_\BB{A}^\BF{m}\times D_\BB{A}}$ and $1_{D_{\BB{X}}}\circ \gamma=1_{D_\BB{A}^\BF{m}\times D_\BB{A}}$ and
such that $\tau\gamma=1_W$;
\item natural transformations $\rho:(U^\BF{n}\times U^\BF{m}) \times U \to U$ and $\sigma :U^\BF{m}\times (U\times U) \to U^\BF{n}$ with components at each $X$ in $\BB{C}$ making the diagrams (\ref{generalized addition axiom}), (\ref{generalized subtractions axiom}) and 
\begin{equation}
\label{generalized protomodular condition}
\vcenter{
\xymatrix{
U(X)^\BF{m}\times(U(X)\times U(X))\ar[dr]^{\pi_1\pi_2}\ar[d]_{\langle \langle \sigma,\pi_1\rangle, \pi_2\pi_2\rangle}&\\
(U(X)^\BF{n}\times U(X)^\BF{m})\times U(X) \ar[r]_(0.7){\rho_X} & U(X)
}
}
\end{equation}
commute.
\end{enumerate}
\end{cor}
\begin{cor}
\label{strange condition cor}
Each of the following types of data uniquely determine each other:
\begin{enumerate}[(a)]
\item natural transformations $\tau: V\to W$ and $\gamma:W\to V$ with $1_{D_{\BB{X}}}\circ \tau=1_{D_\BB{A}^\BF{m}\times D_\BB{A}}$ and $1_{D_{\BB{X}}}\circ \gamma=1_{D_\BB{A}^\BF{m}\times D_\BB{A}}$ and
such that $\gamma\tau=1_V$;
\item natural transformations $\rho:(U^\BF{n}\times U^\BF{m}) \times U \to U$ and $\sigma:U^\BF{m}\times (U\times U)  \to U^\BF{n}$ with components at each $X$ in $\BB{C}$ making the diagrams (\ref{generalized addition axiom}), (\ref{generalized subtractions axiom}) and 
\begin{equation}
\label{strange condition}
\vcenter{
\xymatrix{
(U(X)^\BF{n}\times U(X)^\BF{m})\times U(X)\ar[d]_{\langle \pi_2\pi_1,\langle \rho_X,\pi_2\rangle \rangle}\ar[dr]^{\pi_1\pi_1}\\
U(X)^\BF{m}\times (U(X)\times U(X)) \ar[r]_{\sigma_X}& U(X)^\BF{n}
}
}
\end{equation}
commute.
\end{enumerate}
\end{cor}
\begin{cor}
\label{protomodular + strange condition cor}
Each of the following types of data uniquely determine each other:
\begin{enumerate}[(a)]
\item natural transformations $\tau:V\to W$ and $\sigma:W\to V$ with $1_{D_{\BB{X}}}\circ \tau=1_{D_\BB{A}^\BF{m}\times D_\BB{A}}$ and $1_{D_{\BB{X}}}\circ \gamma=1_{D_\BB{A}^\BF{m}\times D_\BB{A}}$ and inverse to each other;
\item natural transformations $\rho:(U^\BF{n}\times U^\BF{m}) \times U \to U$ and $\sigma:U^\BF{m}\times (U\times U)  \to U^\BF{n}$ with components at each $X$ in $\BB{C}$ making the diagrams (\ref{generalized addition axiom}), (\ref{generalized subtractions axiom}), (\ref{generalized protomodular condition}) and (\ref{strange condition})
commute.
\end{enumerate}

\end{cor}\ \\
We now consider the case where $\BF{m}=\emptyset$ and $\BF{n}=\{1\}$, the results proved here will be used in Section \ref{pointed varieties}.\\

When $\BF{m}=\emptyset$ and $\BF{n}=\{1\}$, the functors $V$ and $W$ are up to isomorphism the functors $\ti{V},\ti{W} : \Pt{A}\to\Pt{X}$ sending $(A,B,\alpha,\beta)$ to 
\[
((U(A)\pb{U(\alpha)}{\theta_B}1)\times U(B),U(B),\pi_2,\langle \langle \theta_A,1\rangle!_{U(B)} , 1 \rangle)
\]
and
\[
(U(A),U(B),U(\alpha),U(\beta))
\]
respectively.
\begin{cor}
\label{thm varphi and rho m=0 and n=1}
Each of the following types of data uniquely determine each other:
\begin{enumerate}[(a)]
\item \label{V to W m=0 and n=1}
a natural transformation $\tau:\ti{V}\to \ti{W}$ with $1_{D_{\BB{X}}}\circ \tau=1_{D_{\BB{A}}}$ and with component at each $(A,B,\alpha,\beta)$ in $\Pt{A}$
such that the diagram
\begin{equation}
\label{kernel identity varphi}
\vcenter{
\xymatrix{
U(A)\pb{U(\alpha)}{\theta_B} 1 \ar[r]^{\langle 1,\theta_B!\rangle} \ar@{=}[d]& (U(A)\pb{U(\alpha)}{\theta_B} 1)\times U(B)\ar[d]^{\varphi_{1_{(A,B,\alpha,\beta)}}} \\
U(A)\pb{U(\alpha)}{\theta_B} 1 \ar[r]_{\pi_1} & U(A)
}
}
\end{equation}
commutes;
\item \label{generalized addition  m=0 and n=1}
a natural transformation $\rho:U \times U \to U$ with component at each $X$ in $\BB{A}$ making the diagram
\begin{equation}
\label{addition axiom}
\vcenter{
\xymatrix{
U(X)\ar[rd]^{1_{U(X)}}\ar[d]_{\langle 1,\theta_{X}!\rangle} & \\
U(X)\times U(X)\ar[r]^(0.7){\rho_X}&U(X)\\
U(X)\ar[ru]_{1_{U(X)}}\ar[u]^{\langle \theta_{X}!,1\rangle}&
}
}
\end{equation}
 commute.
\end{enumerate}
\end{cor}
\begin{cor}
\label{thm psi and sigma m=0 and n=1}
Each of the following types of data uniquely determine each other:
\begin{enumerate}[(a)]
\item \label{W to V m=0 and n=1}
a natural transformation $\gamma:\ti{W}\to \ti{V}$ with $1_{D_{\BB{X}}}\circ \gamma=1_{D_{\BB{A}}}$ and with component at each $(A,B,\alpha,\beta)$ in $\Pt{A}$ 
such that the diagram
\begin{equation}
\label{kernel identity psi}
\vcenter{
\xymatrix{
U(A)\pb{U(\alpha)}{\theta_B} 1 \ar[r]^{\langle 1,\theta_B!\rangle} \ar@{=}[d]& (U(A)\pb{U(\alpha)}{\theta_B} 1)\times U(B)\ar@{<-}[d]^{\psi_{1_{(A,B,\alpha,\beta)}}} \\
U(A)\pb{U(\alpha)}{\theta_B} 1 \ar[r]_{\pi_1} & U(A)
}
}
\end{equation}
commutes;
\item \label{generalized subtractions m=0 and n=1}
a natural transformation $\sigma: U\times U  \to U$ with component at each $X$ in $\BB{A}$ making the diagram
\begin{equation}
\label{subtraction axiom}
\vcenter{
\xymatrix{
U(X)\ar[dr]^{1_{U(X)}}\ar[d]_{\langle 1,\theta_{X}!\rangle} &\\
U(X)\times U(X) \ar[r]_(0.6){\sigma_X} &U(X)\\
U(X) \ar[r]_{!_{U(X)}} \ar[u]^{ \langle 1,1\rangle}& 1 \ar[u]_{\theta_X} 
}
}
\end{equation}
 commute.
\end{enumerate}
\end{cor}
\begin{cor}
\label{thm right omega loop1}
Each of the following types of data uniquely determine each other:
\begin{enumerate}[(a)]
\item 
natural transformations $\tau:\ti{V}\to \ti{W}$ and $\gamma:\ti{W}\to \ti{V}$ with $1_{D_{\BB{X}}}\circ \tau=1_{D_{\BB{A}}}$ and $1_{D_{\BB{X}}}\circ \gamma=1_{D_{\BB{A}}}$ inverse to each other and with components at each $(A,B,\alpha,\beta)$ in $\Pt{A}$ 
making the diagrams (\ref{kernel identity varphi}) and (\ref{kernel identity psi}) commute;
\item 
natural transformations $\rho:U \times U \to U$ and $\sigma: U\times U  \to U$ with component at each $X$ in $\BB{A}$ making the diagrams (\ref{addition axiom}), (\ref{subtraction axiom}), 
\begin{equation}
\vcenter{
\xymatrix{
U(X)\times U(X) \ar[d]_{\langle \sigma_X,\pi_2\rangle} \ar[dr]^{\pi_1} & \\
U(X)\times U(X) \ar[r]_{\rho_X} & U(X)
}
}
\end{equation}
and
\begin{equation}
\vcenter{
\xymatrix{
U(X)\times U(X) \ar[d]_{\langle \rho_X,\pi_2\rangle} \ar[dr]^{\pi_1} & \\
U(X)\times U(X) \ar[r]_{\sigma_X} & U(X)
}
}
\end{equation}
 commute.
\end{enumerate}
\end{cor}
In the sections that follows we use the fact that the set of natural transformation $U^{\BF{n}}\to U$ (where $\BF{n}=\{1,\dots,n\}$ and $U$ is the forgetful functor from a variety to sets) is in bijection with the set of $n$-ary terms of the variety.  Since this is no longer true for arbitrary internal varieties (every term determines a natural transformation but not conversely) the results in the sections that follow hold only partially in arbitrary internal varieties, i.e. the existence of certain terms determine natural transformations between appropriate $V$ and $W$ but not conversely. 
\section{Pointed varieties}
\label{pointed varieties}
In this section we apply the results from Section \ref{general theory} to the special case where $\BB{A}=\V$ is a pointed variety, $\BB{X}=\Set_*$ is the category of pointed sets, $U$ is the usual forgetful functor, $\BF{m}=\emptyset$, $\BF{n}=\{1\}$, and $\theta$ is constructed using the constant of $\V$.

For any category $\BB{C}$ we define $\SE{\BB{C}}$ to be the category of split extensions: an object is a sextuple $(K,A,B,\kappa,\alpha,\beta)$ where $K$, $A$ and $B$ are objects in $\BB{C}$ and $\kappa:K\to B$, $\alpha :A\to B$ and $\beta : B\to A$ are morphisms in $\BB{C}$ with $(K,\kappa)$ the kernel of $\alpha$ and $\alpha\beta =1_B$; a morphism $(K,A,B,\kappa,\alpha,\beta) \to (K',A',B',\kappa',\alpha',\beta')$ is a triple $(u,v,w)$ of morphisms $u:K\to K'$, $v:A\to A'$ and $w:B\to B'$ such that in the diagram
\[
\xymatrix{
K\ar[r]^{\kappa}\ar[d]_{u}&A\ar@<0.5ex>[r]^{\alpha}\ar[d]_{v} & B\ar@<0.5ex>[l]^{\beta} \ar[d]^{w}\\
K'\ar[r]_{\kappa'}&A'\ar@<0.5ex>[r]^{\alpha'} & B'\ar@<0.5ex>[l]^{\beta'} 
}
\]
 $v\kappa=\kappa' u$, $\alpha'v=w\alpha$ and $v\beta=\beta'w$.
\begin{thm}
\label{pointed variety classification}
Let $\cal{V}$ be a pointed variety and let $P,Q: \SE{\cal{V}} \to \SE{\Set_*}$ be the functors taking $(K,A,B,\kappa,\alpha,\beta)$ to $(U(K),U(K)\times U(B),\langle 1,0\rangle, \pi_2,\langle 0,1\rangle)$ and $(U(K),U(A),U(B),U(\kappa),U(\alpha),U(\beta))$ respectively. 
\begin{enumerate}[(a)]
\item \label{unital}$\cal{V}$ is a unital variety \cite{Bourn 1996} if and only if there exists a natural transformation $P \to Q$ with component at $(K,A,B,\kappa,\alpha,\beta)$ of the form 
\[
\xymatrix{
U(K)\ar[r]&U(K)\times U(B)\ar@<0.5ex>[r]^{\pi_2}\ar@{-->}[d]&U(B)\ar@<0.5ex>[l]^{\langle 0,1\rangle}\\
U(K)\ar[r]\ar@{=}[u]&U(A)\ar@<0.5ex>[r]^{U(\alpha)}&U(B);\ar@<0.5ex>[l]^{U(\beta)}\ar@{=}[u]
}
\]
\item \label{subtractive}$\cal{V}$ is a subtractive variety \cite{Janelidze Z 2005} if and only if there exists a natural transformation  $Q \to P$  with component at $(K,A,B,\kappa,\alpha,\beta)$ of the form 
\[
\xymatrix{
U(K)\ar[r]\ar@{=}[d]&U(A)\ar@<0.5ex>[r]^{U(\alpha)}\ar@{-->}[d]&U(B)\ar@<0.5ex>[l]^{U(\beta)}\ar@{=}[d]\\
U(K)\ar[r]&U(K)\times U(B)\ar@<0.5ex>[r]^{\pi_2}&U(B);\ar@<0.5ex>[l]^{\langle 0,1\rangle}
}
\]
\item \label{right omega loop}$\cal{V}$ is a variety of right $\Omega$-loops if and only if there exists a natural isomorphism $P\to Q$ with component at $(K,A,B,\kappa,\alpha,\beta)$ of the form 
\[
\xymatrix{
U(K)\ar[r]&U(K)\times U(B)\ar@<0.5ex>[r]^{\pi_2}\ar@{-->}[d]&U(B)\ar@<0.5ex>[l]^{\langle 0,1\rangle}\\
U(K)\ar[r]\ar@{=}[u]&U(A)\ar@<0.5ex>[r]^{U(\alpha)}&U(B).\ar@<0.5ex>[l]^{U(\beta)}\ar@{=}[u]
}
\]
\end{enumerate}
\end{thm}
\begin{proof}
It is easy to see that to give a natural transformation $P\to Q$ as in (a) above is the same as to give a natural transformation $\tilde{V}\to \tilde{W}$ as in (a) of Corollary \ref{thm varphi and rho m=0 and n=1} which, by Corollary \ref{thm varphi and rho m=0 and n=1}, is uniquely determined by a natural transformation $\rho:U\times U\to U$ with components making the diagram (\ref{addition axiom}) commute.  And, such a natural transformation determines and is determined by a binary term $+$ such that for each $x,y$ in $X$, an algebra, $x+y=\rho_X(x,y)$.  The commutativity of (\ref{addition axiom}) then implies that $x+0=x=0+x$.
The statements (b) and (c) follow from Corollaries \ref{thm psi and sigma m=0 and n=1}, and \ref{thm right omega loop1} in a similar way.
\end{proof}
\begin{cor}
\label{main classification of right omega loops}
Let $\ti{P},\ti{Q} : \Pt{A} \to \Pt{X}$ be the functors sending $(A,B,\alpha,\beta)$ in $\Pt{A}$ to $(U(K\times B),U(B),U(\pi_2), U(\langle 0,1\rangle))$  (where $K=\emph{Ker}(\alpha)$) and  $(U(A),U(B),U(\alpha),U(\beta))$ respectively.  $\V$ is a variety of right $\Omega$-loops if and only if 
there exists a natural bijection $\ti{P} \to \ti{Q}$ with component $(A,B,\alpha,\beta)$ of the form
\[
\xymatrix{
U(K\times B) \ar@<0.5ex>[r]^{U(\pi_2)} \ar@{-->}[d]& U(B)\ar@<0.5ex>[l]^{U(\langle 0, 1\rangle)} \ar@{=}[d]\\
U(A) \ar@<0.5ex>[r]^{U(\alpha)} & U(B).\ar@<0.5ex>[l]^{U(\beta)}
}
\]
\end{cor}
\begin{proof}
It follows from Corollary \ref{protomodular + strange condition cor} that a natural bijection $\ti{P} \to \ti{Q}$ as above is completely determined by and determines binary terms 
$\rho(x,y)$ and $\sigma(x,y)$ satisfying the identities $\sigma(x,x)=0$, $\rho(\sigma(x,y),y)=x$ and $\sigma(\rho(x,y),y)=x$.  Setting $x+y=\rho(\sigma(x,0),y)$ and $x-y=\rho(\sigma(x,y),0)$ determines terms that satisfy the right loop identities.
\end{proof}
\begin{rem}
In fact it can be shown that $\V$ is a variety of right $\Omega$-loops if and only if there exists a natural isomorphism $\ti{P}\to \ti{Q}$.
\end{rem}
\section{Protomodular varieties}
\label{varieties with constants}
In this section we give a new classification of protomodular varieties by  applying the results from Section \ref{general theory} to the case where $\BB{A}=\V$ is an arbitrary variety with constants, $\BB{X}=\Set$ is the category of sets, and $U$ is the usual forgetful functor.
\begin{thm}
$\V$ is a protomodular variety if and only if for some $\BF{m}=\{1,\dots,m\}$, $\BF{n}=\{1,\dots,n\}$ and $\theta$ there exist natural transformations $\tau:V\to W$ and $\gamma:W\to V$ with  $\tau\gamma =1_W$ and with components at each $(A,B,\alpha,\beta)$ in $\Pt{C}$ of the form
\[
\xymatrix@C=15ex{
(U(A)^\BF{n}\pb{U(\alpha)^\BF{n}}{\theta_B} U(B)^\BF{m}) \times U(B)\ar@<0.5ex>[r]^(0.6){\pi_2\times 1}\ar@{-->}@<0.5ex>[d] & U(B)^\BF{m} \times U(B)\ar@<0.5ex>[l]^(0.4){\langle U(\beta)^\BF{m} \theta_B,1\rangle \times 1}\ar@{=}[d] \\
U(B)^m\times U(A) \ar@<0.5ex>[r]^{1\times U(\alpha)}\ar@{-->}@<0.5ex>[u] & U(B)^\BF{m}\times U(B).\ar@<0.5ex>[l]^{1\times U(\beta)}
}
\]
\end{thm}
\begin{proof}
It follows from Corollary \ref{protomodular cor} that natural transformations $\tau:V\to W$ and $\gamma:W\to V$ as above determine terms \[\rho(x_1,\dots,x_n,y_1,\dots,y_m,z) \textrm{ and }\sigma_i(y_1,\dots,y_m,x,z)\ i\in\BF{n}\] satisfying the identities
\[\sigma_i(y_1,\dots,y_m,x,x)=\theta_i(y_1,\dots,y_m)\ i\in\BF{n}\]
\[\rho(\sigma_1(y_1,\dots,y_m,x,z),\dots,\sigma_n(y_1,\dots,y_m,x,z),y_1,\dots,y_m,z)=x.\]  For any constant $e$ we may form new terms $e_i=\theta_i(e,\dots,e)$ $i\in\BF{n}$, $s_i(x,z)=\sigma_i(e,\dots,e,x,z)$ $i\in\BF{n}$, and $p(x_1,\dots,x_n,z)=\rho(x_1,\dots,x_n,e,\dots,e,z)$.  It easy to check that these terms make $\V$ a protomodular variety.  The converse follows from Corollary \ref{protomodular cor} with $\BF{m}=\emptyset$.
\end{proof}
\begin{rem}
The results in this section can easily be extended to $\V$ an infinitary variety, with $\BF{m}$ and $\BF{n}$ possibly infinite sets, giving, by Theorem 2.1 of \cite{Gran Rosicky 2004}, a new classification of infinitary protomodular varieties.
\end{rem}
\begin{rem}
It could also be interesting to study when $\gamma\tau =1_V$ (without $\tau\gamma=1_W$) which can be seen to be equivalent to the existence of $\rho$ and $\sigma$ as above, satisfying the identities:
\begin{eqnarray*}
\sigma_i(y_1,\dots,y_m,x,x)=\theta_i(y_1,\dots,y_m)\ i\in\BF{n}\\
\rho(\theta_1(y_1,\dots,y_m),\dots,\theta_n(y_1,\dots,y_m),y_1,\dots,y_m,x)=x\\
\sigma_i(y_1,\dots,y_m,\rho(x_1,..,x_n,y_1,..,y_m,z) ,z)=x_i\ i \in \BF{n}
\end{eqnarray*}
instead.
\end{rem}
\section{General varieties}
\label{general varieties}
In this section we consider the case where $\BB{A}=\V$ is a variety, $\BB{X}=\Set$ is the category sets, and $U$ is the usual forgetful functor.

For a variety $\V$ consider the condition:
\begin{con}
\label{p and q}
There exist ternary terms $p$ and $q$ satisfying the identities:
$p(x,x,y)=y$ and $p(q(x,y,z),z,y)=x=q(p(x,y,z),z,y)$.
\end{con}
It is easy to see that $q(x,x,y)=y$ follows from the conditions above, as remarked in \cite{Mal'tsev 1954}, where such a variety was called a \emph{biternary system}.
\begin{rem}
It is easy to see that if a variety $\V$ satisfies Condition \ref{p and q} then every regular epimorphism $f:E\to B$ is up to bijection a product projection $\pi_2: X\times B\to B$ for some $X$ (since  for each $b$ and $b'$ choosing $e$ and $e'$ in $f^{-1}(\{b\})$ and $f^{-1}(\{b'\})$ respectively gives a bijection $p(-,e,e'):f^{-1}(\{b\})\to f^{-1}(\{b\})$). 
\end{rem}
\begin{prop}
For a variety $\V$ the following conditions are equivalent:
\begin{enumerate}
\item
$\V$ satisfies Condition \ref{p and q};
\item There exist ternary terms $\ti{p}$ and $\ti{p}$ satisfying the identities:
$\ti{p}(x,x,y)=y=\ti{q}(x,x,y)$, $\ti{p}(x,y,y)=x=\ti{q}(x,y,y)$ and $\ti{p}(\ti{q}(x,y,z),z,y)=x=\ti{q}(\ti{p}(x,y,z),z,y)$;
\item
There exists a quaternary term $u$ satisfying the identities: $u(a,b,b,a)=b$ and
$u(u(a,b,c,d),b,d,c)=a$;
\item
There exists a quaternary term $\ti{u}$ satisfying the identities: $\ti{u}(a,b,b,a)=b=\ti{u}(a,a,b,a)$ and
$\ti{u}(a,b,c,c)=a=\ti{u}(\ti{u}(a,b,c,d),b,d,c)$;
\end{enumerate}
If in addition $\V$ has at least one constant, those conditions are further equivalent to:
\begin{enumerate}
\setcounter{enumi}{4}
\item For each constant $e$ there exist binary terms $x+y$ and $x-y$ satisfying the right loop identities (for that constant $e$).
\end{enumerate}
\end{prop}
\begin{proof}
The implications $2\Rightarrow 1$ and $4\Rightarrow 3$ are trivial.\\
$1\Rightarrow 2 :$ Given $p$ and $q$ define  $\ti{p}(x,y,z)=p(q(x,y,y),y,z)$ and $\ti{q}(x,y,z)=p(q(x,y,z),z,z)$.\\ 
$2\Rightarrow 4 :$ Given $\ti{p}$ and $\ti{q}$ define $\ti{u}(a,b,c,d)=\ti{p}(\ti{q}(a,b,c),d,b)$.\\
$3 \Rightarrow 1:$ Given $u$ define $p(x,y,z)=u(x,z,z,y)$ and $q(x,y,z)=u(x,y,z,y)$.
If in addition $\V$ has at least one constant.\\
$2\Rightarrow 5 :$ Given $\ti{p}$ and $\ti{q}$ for each constant $e$ define $x+y=\ti{p}(x,e,y)$ and $x-y=\ti{q}(x,y,e)$.\\
$5\Rightarrow 1 :$ Given $x+y$ and $x-y$ for some constant $e$ define $p(x,y,z)=q(x,y,z)= (x-y)+z$.
\end{proof}
\begin{rem}
It follows that a variety satisfying Condition \ref{p and q} is a Mal'tsev variety.
\end{rem}
\begin{thm}
\label{p and q thm1}
\begin{enumerate}[(a)]
\item If $\V$ satisfies Condition \ref{p and q}, then for $\BF{n}=\{1\}$, $\BF{m}=\{1\}$ and $\theta=1_{U}$ there exists a natural isomorphism $\tau:V\to W$ with component at each $(A,B,\alpha,\beta)$ in $\Pt{C}$ of the form
\[
\xymatrix@C=15ex{
(U(A)^\BF{n}\pb{U(\alpha)^\BF{n}}{\theta_B} U(B)^\BF{m}) \times U(B)\ar@<0.5ex>[r]^(0.6){\pi_2\times 1}\ar@{-->}[d] & U(B)^\BF{m} \times U(B)\ar@<0.5ex>[l]^(0.4){\langle U(\beta)^\BF{m} \theta_B,1\rangle \times 1}\ar@{=}[d] \\
U(B)^\BF{m}\times U(A) \ar@<0.5ex>[r]^{1\times U(\alpha)}& U(B)^\BF{m}\times U(B);\ar@<0.5ex>[l]^{1\times U(\beta)}
}
\]
\item If for some $\BF{n}=\{1,\dots,n\}$, $\BF{m}=\{1,\dots,m\}$ and $\theta$ there exists a natural isomorphism $\tau:V\to W$ with component at each $(A,B,\alpha,\beta)$ in $\Pt{C}$ of the form 
\[
\xymatrix@C=15ex{
(U(A)^\BF{n}\pb{U(\alpha)^\BF{n}}{\theta_B} U(B)^\BF{m}) \times U(B)\ar@<0.5ex>[r]^(0.6){\pi_2\times 1}\ar@{-->}[d] & U(B)^\BF{m} \times U(B)\ar@<0.5ex>[l]^(0.4){\langle U(\beta)^\BF{m} \theta_B,1\rangle \times 1}\ar@{=}[d] \\
U(B)^\BF{m}\times U(A) \ar@<0.5ex>[r]^{1\times U(\alpha)}& U(B)^\BF{m}\times U(B),\ar@<0.5ex>[l]^{1\times U(\beta)}
}
\]
then $\V$ satisfies Condition \ref{p and q}.
\end{enumerate}
\end{thm}
\begin{proof}
\begin{enumerate}[(a)]
\item Let $\BF{n}=\BF{m}=\{1\}$ and $\theta=1_{U}$.  Given ternary terms $p$ and $q$ as in Condition \ref{p and q}, it is easy to check that $\rho=p$ and $\sigma(x,y,z)=q(y,z,x)$ define natural transformations making the diagrams (\ref{generalized addition axiom}), (\ref{generalized subtractions axiom}), (\ref{generalized protomodular condition}) and (\ref{strange condition}) commute.  Therefore by Corollary \ref{protomodular + strange condition cor} determine a natural isomorphism $V\to W$, as required.
\item If for some $\BF{n}=\{1,\dots,n\}$, $\BF{m}=\{1,\dots,m\}$ and $\theta$ there exists an isomorphism $V\to W$ then by Corollary \ref{protomodular + strange condition cor} there exist terms $\rho(x_1,\dots,x_n,y_1,\dots,y_m,z)$ and $\sigma_i(y_1,\dots,y_m,x,z)$ $i\in \BF{n}$ satisfying the identities:
\begin{eqnarray*}
&\sigma_i(y_1,\dots,y_m,x,x)=\theta_i(y_1,\dots,y_m) &\\
&\rho(\sigma_1(y_1,\dots,y_m,x,z),\dots,\sigma_n(y_1,\dots,y_m,x,z),y_1,\dots,y_m,z)=x &\\
&\sigma_i(y_1,\dots,y_m,\rho(x_1,..,x_n,y_1,..,y_m,z) ,z)=x_i. &
\end{eqnarray*}  Let $p$ and $q$ be the terms defined by 
\begin{eqnarray*}
& p(x,y,z)=\rho(\sigma_1(y,\dots,y,x,y),\dots,\sigma_n(y,\dots,y,x,y),y,\dots,y,z) &\\
& q(x,y,z)=\rho(\sigma_1(z,\dots,z,x,y),\dots,\sigma_n(z,\dots,z,x,y),z,\dots,z,z).&
\end{eqnarray*}
It is easy to check that $p$ and $q$ satisfy the desired identities as in Condition \ref{p and q}.
 \end{enumerate}
\end{proof}
Recall that for any category $\BB{C}$ the functor $D_{\BB{C}}$ is the functor $\Pt{C} \to \BB{C}$ taking $(A,B,\alpha,\beta)$ to $B$. 
\begin{thm}
\label{generalization of introduction equivalent formulations}
Let $V$ and $W$ be the functors defined in Section \ref{general theory} with $\BB{A}=\V$, $\BB{X}=\Set$, $U$ the usual forgetful functor, $\BF{n}=\BF{m}=\{1\}$ and $\theta=1_{U}$. Let $P,Q:(\BB{A}\downarrow D_\BB{A})\to \Pt{X}$ be the functors sending $(E,(A,B,\alpha,\beta),f)$ to 
\[(U(A\pb{\alpha}{f}E)\times U(B),U(E)\times U(B), U(\pi_2)\times 1,U(\langle \beta f,1\rangle)\times 1)\] and \[(U(E)\times U(A),U(E)\times U(B),1\times U(\alpha),1\times U(\beta))\] respectively.
The following are equivalent:
\begin{enumerate}
\item There exists an isomorphism $\tau:V \to W$ with component at each $(A,B,\alpha,\beta)$ in $\Pt{A}$
of the form
\[
\xymatrix@C=15ex{
(U(A)^\BF{n}\pb{U(\alpha)^\BF{n}}{\theta_B} U(B)^\BF{m}) \times U(B)\ar@<0.5ex>[r]^(0.6){\pi_2\times 1}\ar@{-->}[d] & U(B)^\BF{m} \times U(B)\ar@<0.5ex>[l]^(0.4){\langle U(\beta)^\BF{m} \theta_B,1\rangle \times 1}\ar@{=}[d] \\
U(B)^{\BF{m}}\times U(A) \ar@<0.5ex>[r]^{1\times U(\alpha)}& U(B)^\BF{m}\times U(B);\ar@<0.5ex>[l]^{1\times U(\beta)}
}
\]
\item There exists an isomorphism $\chi:P\to Q$ with component at each $(E,(A,B,\alpha,\beta),f)$ in $(\BB{A}\downarrow D_\BB{A})$ of the form
\[
\xymatrix{
(U(A\pb{\alpha}{f}E)\times U(B) \ar@<0.5ex>[r]^{U(\pi_2)\times 1}\ar@{-->}[d] & U(E)\times U(B) \ar@<0.5ex>[l]^{U(\langle \beta f,1\rangle) \times 1}\ar@{=}[d]\\
U(E)\times U(A) \ar@<0.5ex>[r]^{1\times U(\alpha)} & U(E)\times U(B); \ar@<0.5ex>[l]^{1\times U(\beta)}
}
\]
\item $\V$ satisfies Condition \ref{p and q}.
\end{enumerate}
\end{thm}
\begin{proof}
The  equivalence of $1$ and $3$ follows from Theorem \ref{p and q thm1}.  It is easy to show that $2 \Rightarrow 1$ since $P$ and $Q$ composed with the functor sending $(A,B,\alpha,\beta)$ in $\Pt{A}$ to $(B,(A,B,\alpha,\beta),1_B)$ in $(\BB{A}\downarrow D_\BB{A})$ are up to natural isomorphism the functors $V$ and $W$ respectively.
 We will show that $3 \Rightarrow 2$.
 
 Let $p$ and $q$ be ternary terms as in Condition \ref{p and q}.  It is easy to check that $\chi$ with component at each $(E,(A,B,\alpha,\beta),f)$ defined by $\chi_{(E,(A,B,\alpha,\beta),f)} = (\varphi_{(E,(A,B,\alpha,\beta),f)},1_{U(B)})$ where $\varphi_{(E,(A,B,\alpha,\beta),f)}((a,e),b) = (e,p(a,\beta f(e),\beta(b)))$ is an isomorphism with inverse $\chi^{-1}_{{(E,(A,B,\alpha,\beta),f)}}=(\psi_{(E,(A,B,\alpha,\beta),f)},1_U(B))$ where $\psi_{(E,(A,B,\alpha,\beta),f)}(e,a) = ((q(a,\beta\alpha(a),\beta f(e)),e),\alpha(a))$.
\end{proof}

\end{document}